\documentclass[11pt]{article}
\usepackage{graphicx}
\usepackage[colorlinks=True]{hyperref}

\usepackage{amsmath,amssymb,amsthm}
\usepackage{xcolor}
\usepackage{dsfont}
\usepackage{tikz,mathrsfs}
\usepackage{authblk}
\usepackage{geometry}
\usepackage{yhmath}
\usepackage{comment}
\usepackage{mathtools}
\mathtoolsset{showonlyrefs}
\usepackage{lipsum}

\newcommand\blfootnote[1]{%
  \begingroup
  \renewcommand\thefootnote{}\footnote{#1}%
  \addtocounter{footnote}{-1}%
  \endgroup
}
\usepackage{todonotes}

\newcommand{\R}{\mathbb{R}}

\newcommand{\N}{\mathbb{N}}
\newcommand{\Z}{\mathbb{Z}}
\newcommand{\C}{\mathbb{C}}

\renewcommand{\geq}{\geqslant}
\renewcommand{\leq}{\leqslant}
\DeclareMathOperator\re{Re}

\usepackage{tikz}
\usetikzlibrary{arrows}
\usetikzlibrary{calc}

\DeclareMathOperator\diag{diag}

\newtheorem{thm}{Theorem}
\newtheorem{assump}{Assumption}
\newtheorem{defi}[thm]{Definition}
\newtheorem{lemma}[thm]{Lemma}
\newtheorem{prop}[thm]{Proposition}
\newtheorem{model}{Model}

\DeclareMathOperator\im{Im}
\DeclareMathOperator\e{e}

\DeclareMathOperator{\elliptsc}{sc}
\DeclareMathOperator{\elliptnd}{nd}
\DeclareMathOperator{\elliptsn}{sn}
\DeclareMathOperator{\elliptcn}{cn}
\DeclareMathOperator{\elliptdn}{dn}
\DeclareMathOperator{\elliptsd}{sd}
\DeclareMathOperator{\elliptcd}{cd}
\DeclareMathOperator{\elliptdc}{dc}

\DeclareMathOperator{\gr}{Gr}
\DeclareMathOperator\id{Id}

\title{The limit shape of the Leaky Abelian Sandpile Model on isoradial graphs}

\begin{document}

\date{\today}

\author{Théo Ballu\thanks{Théo Ballu, Univ Angers, CNRS, LAREMA, SFR MATHSTIC, F-49000 Angers, France, \href{mailto:theo.ballu@univ-angers.fr}{theo.ballu@univ-angers.fr}.}}
\maketitle

\begin{abstract}
    We study the limit shape of the boundary of the leaky sandpile model on isoradial graphs. These graphs are equipped with conductances and masses introduced by Boutillier, de Tilière and Raschel, which are defined with the help of the geometry of the graph and Jacobi elliptic functions. Building on the link between the shape of the boundary of the sandpile and the Green function of the graph, together with the asymptotics of the Green function obtained by the previously mentioned authors, we prove an explicit formula for the limit shape, both in $\R^d$ where the graph can be seen as a monotone surface under our main geometric assumption, and on the plane where it naturally lies. We then study the limit regime where the elliptic modulus, the parameter of elliptic functions, tends to $0$, which comes down to approaching the massless case, and obtain a circle as a universal limit shape.
\end{abstract}

\section{Introduction and main results}
\subsection{Context}
\paragraph{Statistical mechanics on isoradial graphs}

Discrete models of statistical mechanics study physical systems using a microscopic description, typically of the size of atoms or molecules. The way to do it is usually to associate random configurations on graphs, where vertices represent particle sites and edges model interactions between particles. The first graphs studied are the lattices $\Z^d$, $d \in \N$, or regular planar graphs such as triangle or hexagonal graphs, but a more general class of planar graphs that are well suited for statistical mechanics has emerged: isoradial graphs (Definition~\ref{def:isoradial_graphs}), see the survey \cite{BdT_survey_isoradial} of Boutillier and de Tilière for a full review. One of the reasons for the popularity of isoradial graphs is the possibility to perform discrete complex analysis, see for example Kenyon in \cite{kenyon_laplace_dirac_planar} or Chelkak and Smirnov in \cite{ChSm}. Many classical models have been studied on isoradial graphs. Two of the most famous ones are the dimer model, see \cite{kenyon_laplace_dirac_planar}, and the Ising model, studied for example in \cite{ChSm} and \cite{BdT}.

In this article, the probability model that is used is the random walk on isoradial graphs. As far as statistical mechanics is concerned, it is well known that it can be related with the spanning tree model (see for example \cite{PrWi-98} for an algorithmic point of view). Isoradial graphs allow to compute local integral formulas for the Green function of graphs equipped with conductances and masses. In \cite{kenyon_laplace_dirac_planar}, Kenyon studies it for critical weights while the article \cite{BdTR} of Boutillier, de Tilière and Raschel, which we will be relying on extensively, generalizes it to a family of non-critical weights. Asymptotics of Green functions are obtained performing an asympotic analysis in the integral formulas, which are used by two of the previous authors in \cite{BR22} to compute the Martin boundary of killed random walks on isoradial graphs.

\paragraph{The sandpile model and its leaky variant}
The Abelian Sandpile Model is a cellular automaton that aims to model the spreading of sand on a graph $G = (V,E)$. It was invented in 1987 in \cite{BTW} by Bak, Tang and Wiesenfeld to illustrate the notion of self-organized criticality, a general property of some physics models. In its most basic form, it starts from a sandpile configuration $s : V \to \Z_{\geq 0}$ that counts the number of grains of sand at each vertex. As long as there are vertices which have more grains of sand than their degree in the graph, they topple, that is, they send one grain of sand to each of their neighbors, until the configuration eventually becomes stable. If at some point, several vertices can topple, then the chosen order does not change the final stable configuration, hence the term \emph{Abelian}. Surveys on the sandpile model can be found, for example, in \cite{Hol_survey} or \cite{Jarai_survey}.

An interesting question about the model is the shape of the final configuration (after rescaling) when the initial configuration consists of $N$ grains of sand at the origin. On the plane, the fractal structure of the interior of the shape has been studied in link with Apollonian circles in \cite{LevPegSma}, but less is known about the boundary of the shape. A proof of convergence has been given by Pegden and Smart in \cite{PegSma} and the case where the initial configuration is random was studied by Bou-Rabee in \cite{Bou21}, but a precise description of the boundary is still lacking. For example, it was conjectured but not proved that it is not a circle.

Recently, explicit descriptions of this boundary were given for an alternative sandpile model, called \emph{leaky}. In this model, the number of grains of sand on a vertex is no longer necessarily an integer, the number of grains given by a vertex is not necessarily the same for all its neighbors and, most important, a fraction of the sand on a vertex is lost each time it topples. A full description in our context is given in Model~\ref{model:sandpile_article3}. 
Using the link between a killed random walk and the shape of the sandpile, Alevy and Mkrtchyan studied in \cite{AM22} the limit shape of the boundary of the leaky Abelian sandpile on $\Z^2$. They prove that it is the dual curve of the boundary of the gaseous phase of the amoeba of a Laurent polynomial. Their work was generalized in \cite{BBMR} to $\Z^d$ for any $d$, and even thickenings of $\Z^d$, that is spaces of the form $\Z^d \times \{1, \ldots,p\}$ where the toppling rules are different on each of the $p$ layers of $\Z^d$, thus providing a convergence result for spatially non-homogeneous models.

\subsection{Main results}\label{subsec:main_results}
\paragraph{General aim of the paper} The goal of this article is to provide a link between Green functions and the shape of the sandpile on isoradial graphs, as it was done in the previously mentioned works \cite{AM22, BBMR} for other graphs, and to use the asymptotics of the Green functions from \cite{BdTR, BR22} to analyze the shape of the leaky sandpile model on isoradial graphs equipped with the masses and conductances of \cite{BdTR}. These masses and conductances are defined using geometric quantities of the graph and Jacobi elliptic functions. 

\paragraph{Geometric assumptions on the graph}
The originality of this work is that the underlying graph is no longer periodic, not even partially. Periodicity is replaced by the geometrical definition of isoradial graphs, with a few additional assumptions. 
These geometric assumptions made on the graph are essentially the same as in \cite{BdTR, BR22}. For most of our results, the graph will be assumed to be quasicrystalline, which means that the edges of the diamond graph (Definition~\ref{def:isoradial_graphs}) take a finite number $d$ of directions. In this setting, an isoradial graph can be seen as the projection onto the plane of a monotone surface $S$ of $\Z^d$, and we will therefore have two kinds of results: results of convergence on the surface $S \subset \R^d$ and on the plane, where the isoradial graph is naturally embedded. 

To obtain a description of the limit shape in spherical coordinates in $\R^d$, we need an additional regularity assumption on the monotone surface associated with the graph. We propose two of such assumptions (see Assumptions~\ref{assump:regularity_weak} and \ref{assump:regularity_strong}) that essentially say that the directions of $\R^d$ (for the $\ell^1$ norm) that appear at infinity on the monotone surface do not go too far away from the surface. One of these assumptions is more restrictive than the other, leading to a stronger convergence result. 

As for convergence results on the plane, we will need to assume that the graph is \emph{asymptotically flat} (Assumption~\ref{defi:asymptotically_flat}). This assumption, that comes from the work \cite{BR22} of Boutillier and Raschel, is another kind of regularity assumption, that allows to associate a direction in $\R^d$ to any direction on the plane. It is stronger that the regularity assumptions \ref{assump:regularity_weak} and \ref{assump:regularity_strong}.

\paragraph{Main theorems}
We study the limit shape of the leaky sandpile model with conductances and masses defined in \eqref{eq:conductance} and \eqref{eq:mass} when the initial configuration consists of $N$ grains at the origin. 
As mentioned before, since we work with quasicrystalline graphs, we have two types of results: on a surface of $\R^d$ and on its projection onto the plane. 
\begin{itemize}
    \item We first study the limit shape of the sandpile when $N$ goes to infinity, after rescaling by $\log N$. In $\R^d$, we provide two theorems that show the convergence to a limit shape, that is described in spherical coordinates in \eqref{eq:limit_shape_N_infinity}. This limit shape is parametrized by an explicit function defined in \eqref{eq:def_theta} using Jacobi elliptic functions. The first theorem, Theorem~\ref{thm:case_N_infinity_weak_regularity_assumption}, provides a convergence under the weak regularity Assumption~\ref{assump:regularity_weak}, while the second one, Theorem~\ref{thm:case_N_infinity_strong_regularity_assumption}, ensures that this convergence is uniform in the directions under the stronger regularity Assumption~\ref{assump:regularity_strong}. On the plane, the \emph{asymptotically flat} assumption ensures that the stronger regularity assumption is true, leading to Theorem~\ref{thm:limit_shape_N_plane}, which states that the sandpile on the plane converges uniformly to the projection of the shape previously obtained in $\R^d$.
    \item We then study the limit case where the elliptic modulus, that is the parameter for Jacobi elliptic functions, tends to $0$ (\emph{after} having $N$ tend to infinity). When the elliptic modulus is equal to $0$, Jacobi elliptic functions become usual trigonometric functions. This allows us to get a limit shape in $\R^d$, easily expressed with trigonometric functions in Proposition~\ref{prop:limit_k_0_surface}. Once projected onto the plane, Proposition~\ref{prop:limit_shape_k_0_plane} shows that this limit shape is universal and does not depend on the graph: it is simply a circle.
\end{itemize}

\subsection{Outline of the paper}
\begin{itemize}
    \item Section~\ref{sec:background_isoradial} reminds the main definitions and results needed from \cite{BdTR} on isoradial graphs with conductances and masses defined with the help of Jacobi elliptic functions. The two principal objects of interest are the Laplacian of the graph and its inverse, the Green function. The main result we will use, which is \cite[Theorem~14]{BdTR}, is reminded in Theorem~\ref{thm:asymp_green}. It gives the asymptotics of the Green function $\gr(x_0,y)$ of the graph (or equivalently of the potential function of the associated random walk) as $y$ goes to infinity.
    \item Section~\ref{sec:sandpile_model} describes the leaky sandpile model on our isoradial graph equipped with the previously mentioned conductances and masses. The final shape of the sandpile started with $N$ grains, for $N$ fixed, is defined with the help of the odometer function. Using similar ideas to those of \cite{AM22, BBMR}, that is applying the (modified) Laplacian of the graph to the odometer function, we obtain the main result of this section, Proposition~\ref{prop:threshold}. In this proposition, we show that there are two thresholds, of the form $\frac{c}{N}$ for some constants, such that when the Green function at a point $x$, $\gr(x_0,x)$, is above one of those thresholds, the point $x$ is in the stable configuration of the sandpile started with $N$ grains of sand at $x_0$, while if it is under the other threshold, it is not in the configuration. In other words, the boundary of the sandpile is bounded between two level curves of the Green functions, for levels of order $\frac{1}{N}$, allowing to use the asymptotics of the Green function mentioned in Section~\ref{sec:background_isoradial} to get approximations of the shape of this boundary.
    \item Finally, in Section~\ref{sec:limit_shape}, we state and prove the main convergence results mentioned in Section~\ref{subsec:main_results}. The goal of Section~\ref{subsec:shape_r^d} is to study the limit shape of the boundary of the sandpile in $\R^d$. To do so, we rely on the analysis of \cite{BBMR}, which studies the case of $\Z^d$. However the results of \cite{BBMR} are not enough, as in the case of quasicrystalline graphs, the sandpile does not live in the whole of $\R^d$, but on a monotone surface, hence the need for the regularity assumptions on the surface mentioned in Section~\ref{subsec:main_results}. With these assumptions, we are able to state two convergence theorems that describe the limit shape of the boundary of the sandpile in spherical coordinates, Theorems~\ref{thm:case_N_infinity_weak_regularity_assumption} and \ref{thm:case_N_infinity_strong_regularity_assumption}. We then study what this limit shape becomes when the elliptic modulus, that parameterizes Jacobi elliptic functions, tends to $0$, see Proposition~\ref{prop:limit_k_0_surface}.

    Using the \emph{asymptotically flat} assumption of \cite{BR22} and properties of the projection from the surface of $\R^d$ onto the plane, we finally study the two previous regimes, that is $N$ to infinity and $k$ to $0$ (after having $N$ tend to infinity), for the natural embedding on the plane of the isoradial graph, obtaining an explicit description in spherical coordinates of the limit shape on the plane in Theorem~\ref{thm:limit_shape_N_plane}. When the elliptic modulus tends to $0$, this limits shapes turns out to be a circle, as seen in Proposition~\ref{prop:limit_shape_k_0_plane}.
    
\end{itemize}

\section{Background on isoradial graphs and random walks}\label{sec:background_isoradial}
In this section, we remind general facts on isoradial graphs and the asymptotics of Green functions obtained in \cite{BR22, BdTR}. 
\subsection{Definition, geometric assumptions}

\begin{defi}[Isoradial graph, diamond graph]\label{def:isoradial_graphs}
    An \emph{isoradial graph} $G = (V,E)$ is a planar graph with an embedding such that every face is inscribable in a circle of radius $1$, which center lies in the interior of the face. The dual graph $G^*$ of $G$ is embedded on the plane by placing its vertices at the center of those circles.

    The \emph{diamond graph} $G^\diamond$ associated with $G$ is constructed as follows. The vertices of $G^\diamond$ are those of $G$ and of $G^*$ and the edges of $G^\diamond$ connect every vertex of $G^*$ to the corners of the corresponding face of $G$.
\end{defi}

\begin{figure}[htbp]
    \centering
    \includegraphics[scale=0.12]{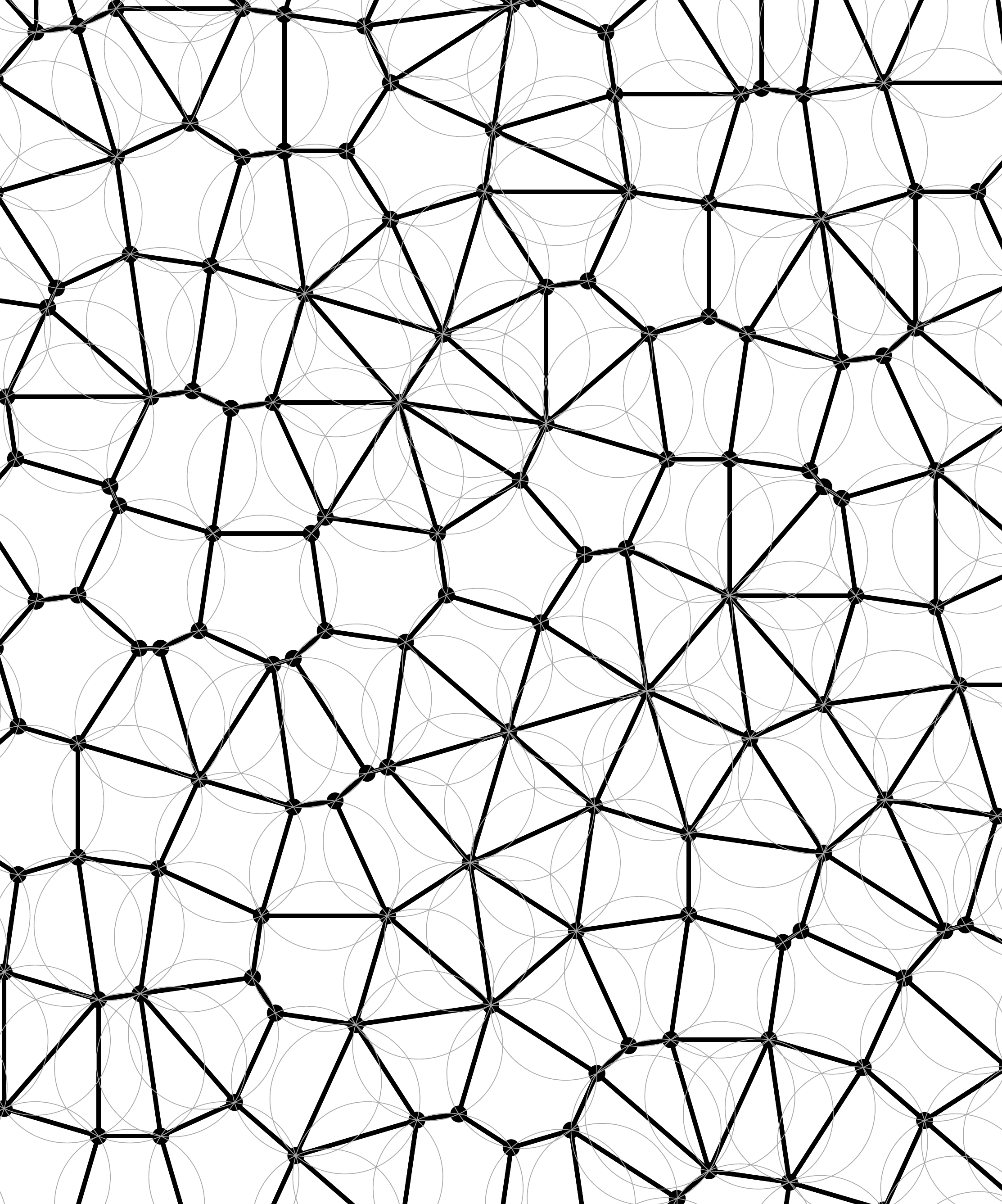}
    \includegraphics[scale=0.12]{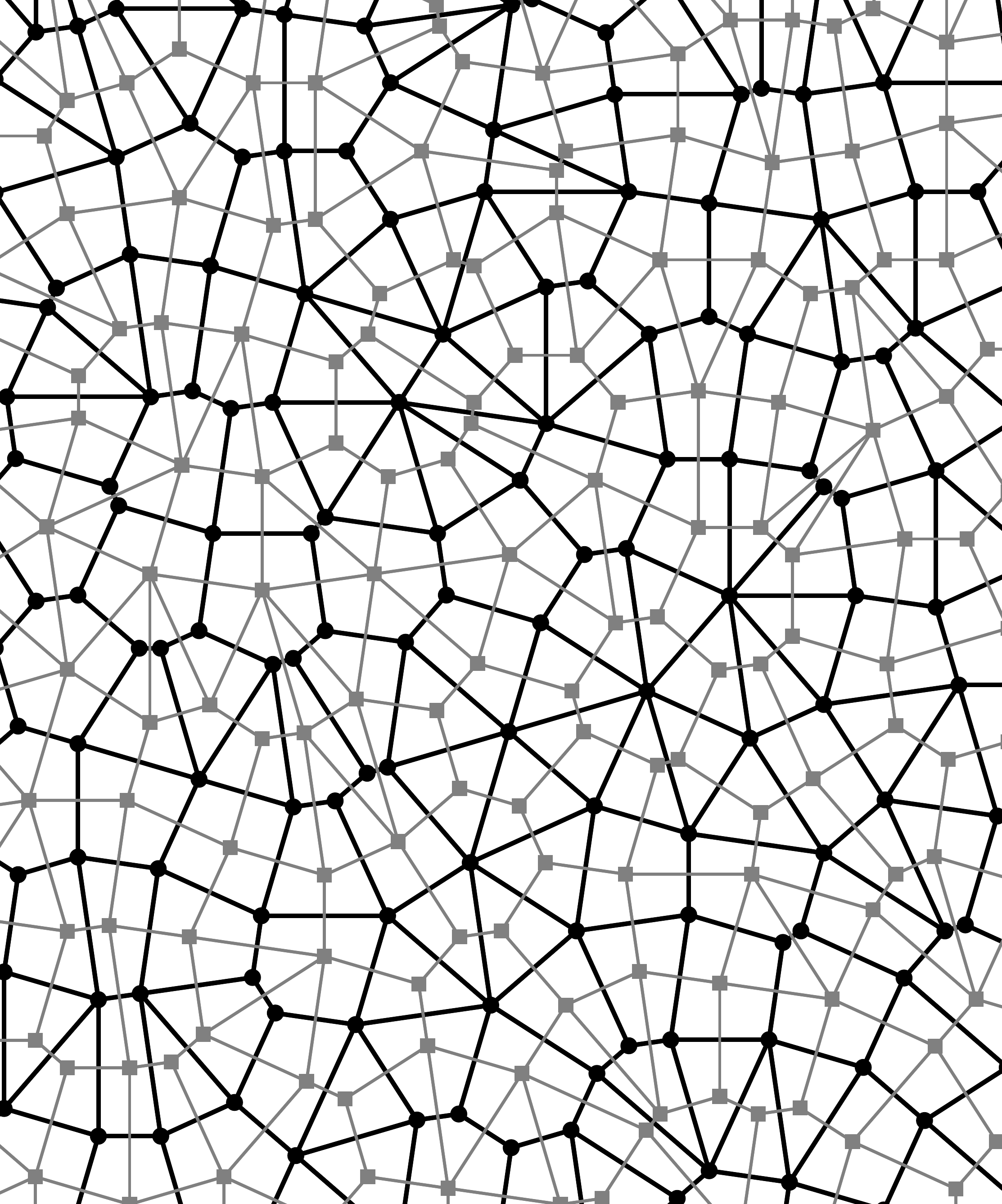}
    \includegraphics[scale=0.12]{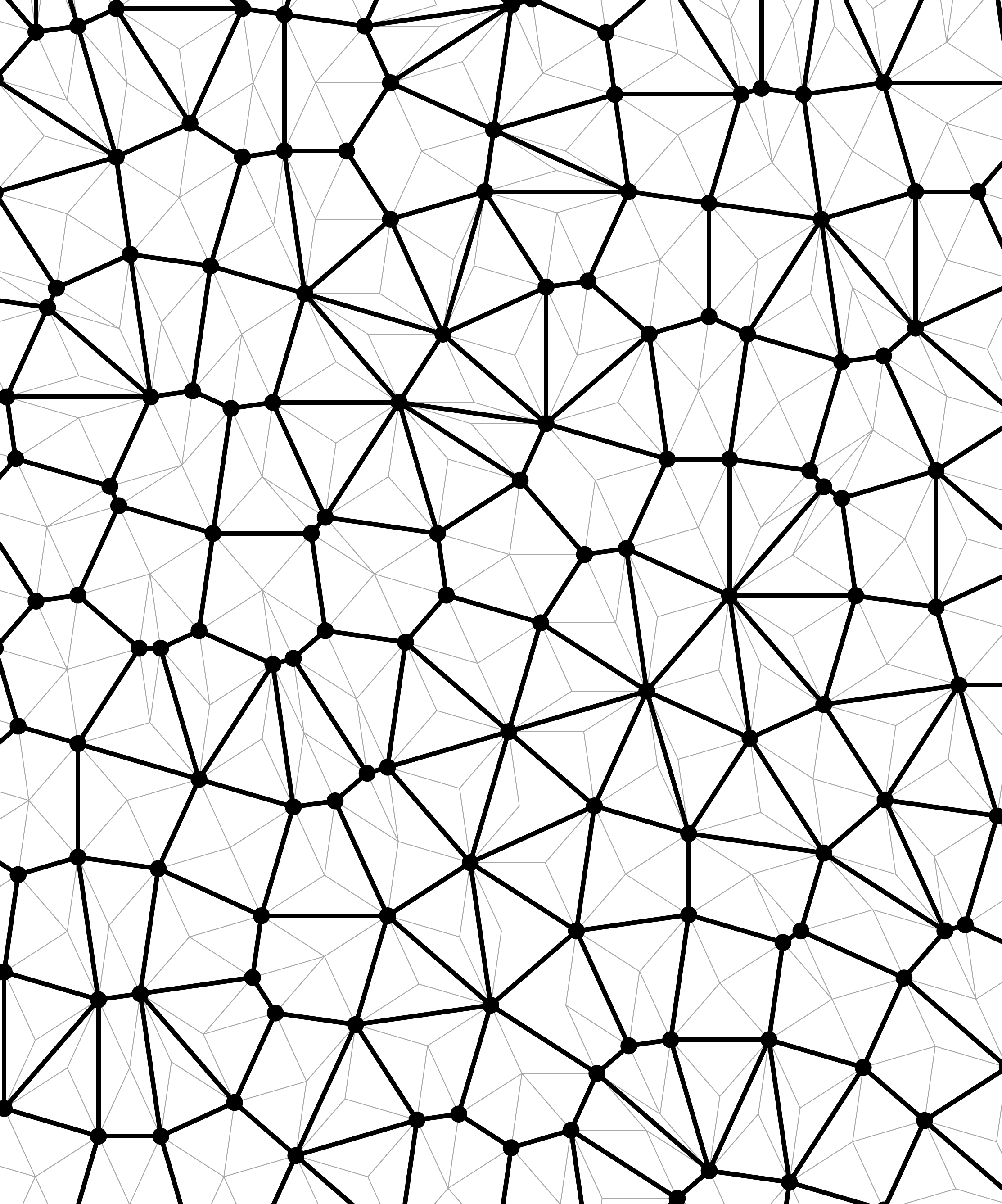}
    \caption{An isoradial graph, with its dual graph, and with its diamond graph. Illustrations by Cédric Boutillier.}
    \label{fig:my_label}
\end{figure}

From now on, $G = (V,E)$ is an isoradial graph whose faces cover the whole plane.
The faces of $G^\diamond$ are rhombi, whose diagonals are edges of $G$ and $G^*$.

The authors of \cite{BR22, BdTR} use the geometry of the diamond graph $G^\diamond$ to define conductances and masses on the graph $G$, thus constructing a killed random walk on $G$. We remind their notations and definitions.  The angle $\overline{\theta}_e$ associated with the edge $e$ is the half angle of the rhombus associated with $e$, as in Figure~\ref{fig:angles_lozanges}.

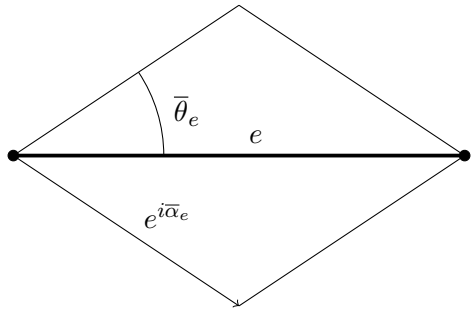
\begin{figure}
    \centering
    \begin{tikzpicture}
        \draw[fill] (0,0) circle (0.07);
        \draw[fill] (6,0) circle (0.07);
         \draw[-, line width = 0.5mm] (0,0) to (6,0);
         \draw[ shift={(0,0.25)}] (3,0) node[right]{$e$};
         \draw[-] (0,0) to (3,2);
         \draw[-] (6,0) to (3,2);
         \draw[->] (0,0) to (3,-2);
         \draw[-] (6,0) to (3,-2);
         \draw (2,0) arc (0:33.69:2);
         \draw[ shift={(0,0)}] (2,.6) node[right]{$\overline{\theta}_e$};
         \draw[ shift={(0,0)}] (1.6,-0.8) node[right]{$e^{i\overline{\alpha}_e}$};
    \end{tikzpicture}
    \caption{Notations for the angles associated with an edge $e$ in the diamond graph.}
    \label{fig:angles_lozanges}
\end{figure}

Here are the main geometric assumptions made.

\begin{assump}\label{assump:angles_not_too_small}
    There exists $\varepsilon>0$ such that, for every edge $e$, $\theta_e \in \left[\varepsilon, \frac{\pi}{2}-\varepsilon\right]$.
\end{assump}

In particular, this assumption ensures that the graph is \emph{locally finite}. A stronger geometric assumption is sometimes necessary. 

\begin{defi}\label{def:quasicrystalline}
    An isoradial graph $G$ is said to be \emph{quasicrystalline} if the number $d$ of directions of the edges of its diamond graph is finite.
\end{defi}
In the quasicrystalline case, we can see the edges of the diamond graph $G^\diamond$ as projections of unit vectors of $\Z^d$. The diamond graph $G^\diamond$ is then the projection onto a plane of a \emph{monotone surface} in $\Z^d$, that will be denoted $S$. See Figure~4 of \cite{BR22} for an illustration when $d=3$. If an origin $x_0$ is fixed in the graph, then for any minimal path from $x_0$ to a vertex $y$, we can count the number of edges $n(y) = (n_1(y), \ldots, n_d(y))$ of each type, counted positively or negatively according to the direction. These are the \emph{coordinates of $y$ when lifted onto the monotone surface of $\Z^d$}. The \emph{reduced coordinates} are $\frac{n(y)}{\|n(y)\|_1}$ where $\|n(y)\|_1 = \sum_{j=1}^d |n_j(y)|$.

When talking about norms or directions in $\R^d$, we will always use the $\ell^1$ norm.

We will study the sandpile model on quasicrystalline isoradial graphs from two points of view: on the plane, where the graph is naturally embedded, and on the surface $S$ of $\Z^d$. If we denote $\overline{\alpha}_1, \ldots, \overline{\alpha}_d$ the angles of the $d$ types of edges, the projection that sends $S$ to $G^\diamond$ is \begin{equation}
    \begin{array}{ccccc}
        \pi & :& S & \longrightarrow &G^\diamond  \\
         & &\left( x_1, \ldots, x_d \right) & \longmapsto & \sum_{j=1}^d x_j e^{i \overline{\alpha}_j}.
    \end{array}
\end{equation}

\subsection{Killed random walks on isoradial graphs and their Green functions}

\paragraph{Jacobi elliptic functions, conductances and masses}
In \cite{BdTR}, the isoradial graph $G$ is given conductances and masses using the geometry of the graph and Jacobi elliptic functions.

We will keep to the bare essentials on Jacobi elliptic function and refer to \cite{abramowitz, Law} and \cite[Section~2.2]{BdTR} for useful properties.
We first fix an \emph{elliptic modulus} $k \in (0,1)$. The complementary elliptic modulus is $k' = \sqrt{1-k^2}$. The complete elliptic integral of the first kind $K$ and of the second kind $E$ are
\begin{equation}
    K := K(k) := \int_{0}^{\frac{\pi}{2}} \frac{1}{\sqrt{1-k^2 \sin^2(\tau)}} \mathrm{d}\tau
\end{equation}
and
\begin{equation}
    E := E(k) := \int_0^{\frac{\pi}{2}} \sqrt{1-k^2 \sin^2(\tau)}\mathrm{d}\tau.
\end{equation}
Their complementary integrals are \[K' := K'(k) := K(k')~~~~\textrm{et}~~~~ E' := E'(k) := E(k').\]
We note $\mathrm{s}$, $\mathrm{c}$, $\mathrm{d}$ et $\mathrm{n}$ the complex numbers $0$, $K$, $K +i K'$ and $i K'$. If $\mathrm{p}$ and $\mathrm{q}$ are two distinct letters of $\{ \mathrm{s,c,d,n}\}$, then $\mathrm{pq} := \mathrm{pq}(\cdot |k)$ is the unique meromorphic, doubly periodic function on $\C$ such that:
\begin{itemize}
    \item $\mathrm{p}$ is a simple zero of $\mathrm{pq}$ and $\mathrm{q}$ is a simple pole;
    \item  $\mathrm{q}- \mathrm{p}$ is a half-period of $\mathrm{pq}$, while vectors joining $\mathrm{p}$ to the two other letters are quarter periods;
    \item the first coefficient of the expansion of $\mathrm{pq}(u |k)$ in power series at $0$ is $1$, in other words the first term is either $u$, $\frac{1}{u}$ or $1$ depending on whether $0$ is a zero, a pole or none of those.
\end{itemize}

\begin{figure}
    \centering
    \begin{tikzpicture}
  \coordinate (s) at (0, 0);   
  \coordinate (c) at (6, 0);   
  \coordinate (d) at (6, 3);   
  \coordinate (n) at (0, 3);   

  \draw (s) -- (c) -- (d) -- (n) -- cycle;

  \foreach \point in {s, c, d, n} {
    \fill (\point) circle[radius=2pt];
  }

  \node[below left] at (s) {s};
  \node[above right] at (s) {$0$};
  \node[below right] at (c) {c};
  \node[above left] at (c) {$K$};
  \node[above right] at (d) {d};
  \node[below left] at (d) {$K + i K'$};
  \node[above left] at (n) {n};
  \node[below right] at (n) {$iK'$};
\end{tikzpicture}
    \caption{The rectangle $[0,K] \times [0, K']$ used to define Jacobi elliptic functions.}
    \label{fig:rectangle_ellipt_func}
\end{figure}
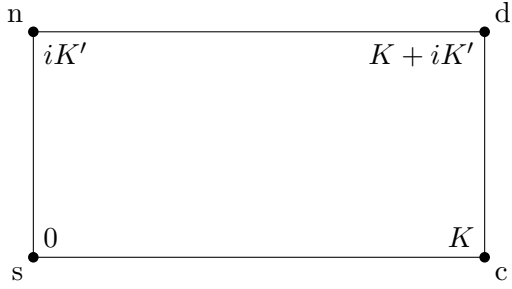

A useful auxiliary function needed in \cite{BdTR} to define the mass is the function $\mathrm{A}$ defined by \[\forall u \in \C,~~ \mathrm{A}(u|k) = \frac{1}{k'} \left( \mathrm{Dc}(u|k) + \frac{E-K}{K}u \right),\]
where
\[\mathrm{Dc}(u|k) = \int_0^u \mathrm{dc}^2(v|k) \mathrm{d}v.\]
When dealing with elliptic functions with modulus $k$, angles have to be multiplied by $\frac{2K(k)}{\pi}$. Following the notations of \cite{BdTR}, natural angles are noted with a bar (e.g. $\overline{\theta}$, $\overline{\alpha}$) while elliptic angles are noted without a bar (e.g. $\theta = \frac{2K}{\pi} \overline{\theta}$, $\alpha = \frac{2K}{\pi} \overline{\alpha}$).

The \emph{conductance} of an edge $e = xy$ is then defined by
\begin{equation}\label{eq:conductance}
    \rho_e = \rho(xy) = \elliptsc{\left( \theta_e | k \right)},
\end{equation}
and if $xy$ is not an edge of $G$, then we set $\rho(xy) = 0$.
The (squared) \emph{mass} of a vertex $x$ linked with edges $e_1 , \ldots, e_k$ is 
\begin{equation}\label{eq:mass}
    m^2(x) = \sum_{j=1}^k \left( A\left( \theta_{e_j} |k \right) -\elliptsc \left( \theta_{e_j} |k \right) \right)
\end{equation}
where 
\begin{equation}
    \forall u \in \C,~~ A(u|k) = \frac{1}{k'} \left( \int_{0}^{u} \elliptdc^2(v|k) \mathrm{d}v  + \frac{E-K}{K} u \right).
\end{equation}
Let us justify the relevance of the definitions of conductances and masses. First of all, they are nonnegative quantities, which is the minimum requirement. This is immediately apparent for conductances, and is proved in \cite[Prop~6]{BdTR} for masses, which are even positive when $k>0$. Furthermore, taking $k = 0$, we obtain $m^2 = 0$ and $\rho_e = \tan\left( \overline{\theta}_e \right)$. We thus find the critical weights studied by Kenyon in \cite{kenyon_laplace_dirac_planar}. The authors of \cite{BdTR} therefore define a family of conductances and masses that are noncritical for $k > 0$ and critical for $k=0$. The analyticity with respect to the parameter $k$ around $0$ (\cite[Lemma~7]{BdTR}) shows that this is not just \emph{any} extension of the critical case. Above all, this choice is relevant because of its compatibility with the geometry of the graph. Indeed, the Laplacian satisfies a principle of 3-dimensional consistency. In practice, this means that harmonic functions are invariant under star-triangle transformations of the graph. More precisely, consider a graph with a star, i.e., a vertex of degree $3$. We can delete this vertex and connect its three neighbors in pairs to obtain a triangle pattern instead of the star. Conversely, starting from a triangle pattern, we can add a vertex to obtain a star (see Figure~\ref{fig:star_triangle} or \cite[Fig~6]{BdTR}). Three-dimensional consistency means that if a function $f$ is harmonic at a vertex $x_0$ for the Laplacian of a graph with a star at $x_0$, then the Laplacians of $f$ for the graph with a star and the associated graph with a triangle are the same. Conversely, if $f$ is a function defined on a graph with a triangle, then there exists a unique extension of $f$ to the associated graph with a star such that the Laplacians are the same. This is the subject of Proposition~8 in \cite{BdTR}. The discussion following this proposition explains how, under certain assumptions, a harmonic function on an isoradial graph can be transformed into a harmonic function on a network of the form $\Z^d$ using a sequence of star-triangle transformations.

\begin{figure}
    \centering
    \includegraphics[scale=0.8]{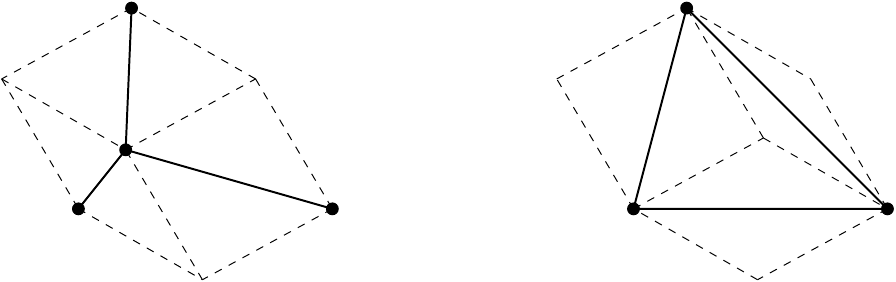}
    \caption{A star-triangle transformation.}
    \label{fig:star_triangle}
\end{figure}

\paragraph{Killed random walk}
These conductances and masses can be used to define a killed random walk $\left( X_n \right)_{n \in \N}$ on $G$.
More precisely, starting from a vertex $x$, the probability to jump to $y$ is
\begin{equation}
    \mathbb{P} \left( X_{n+1} = y |X_n=x \right) =\mathbb{P}(x \to y) = \frac{\rho(xy)}{\sum_{z \in V} \rho(xz) + m^2(x)}
\end{equation}
and the probability to kill the walk at $x$ is 
\begin{equation}
    \mathbb{P}\left( X_{n+1}\text{ is killed} | X_n =x \right) = \mathbb{P}(x \to \text{kill}) = 1- \sum_{z \in V} \mathbb{P}(x \to z) .
\end{equation}


In order to work with the potential of the walk, we need it to be transient, which is a consequence of the killing. We even have a much stronger result: the probability to kill the walk is bounded away from zero.

\begin{lemma}\label{lem:diagonal_laplacian_bound}
    There exist constants $c, c', \delta >0$ such that for every $x \in V$,
    \begin{equation}\label{eq:bounds_diagonal_laplacian}
        c \leq \sum_{z \in V} \rho(xz) + m^2(x) \leq c' ;
    \end{equation}
    \begin{equation}\label{eq:proba_kill_bounded_away_from_zero}
        \mathbb{P}(x \to \textnormal{kill}) > \delta.
    \end{equation}
\end{lemma}

\begin{proof}
According to Assumption~\ref{assump:angles_not_too_small}, each vertex has at most $\frac{\pi}{\varepsilon}$ neighbors in the graph. Besides, the function $\elliptsc(\cdot | k )$ is positive and increasing on $\left( 0, \frac{\pi}{2} \right)$, so for every edge $xy$, $\elliptsc(\varepsilon|k) \leq \rho(xy) \leq \elliptsc\left( \frac{\pi}{2} - \varepsilon|k\right)$, hence \[ \elliptsc(\varepsilon|k) \leq\sum_{z \in V} \rho(xz) \leq \frac{\pi}{\varepsilon} \elliptsc\left( \frac{\pi}{2} - \varepsilon|k\right). \]
Moreover, \cite[Prop 6]{BdTR} proves that the mass $m^2$ is non-negative and using once again Assumption~\ref{assump:angles_not_too_small} to bound $A$, we can similarly bound $m^2$ uniformly in $x$, leading to \eqref{eq:bounds_diagonal_laplacian}.

Given that $\mathbb{P}(x \to \textnormal{kill}) = \frac{m^2(x)}{\sum_{z \in V} \rho(xz) + m^2(x)}$ and that \eqref{eq:bounds_diagonal_laplacian} bounds the denominator, we only need to bound $m^2(x)$ away from zero to get \eqref{eq:proba_kill_bounded_away_from_zero}. To do so, we take a closer look at the proof of the non-negativity of $m^2$ in \cite[Prop 6]{BdTR}. It is proved that the function $ f = A(\cdot |k) - \elliptsc(\cdot|k)$ is strictly concave and takes the value $0$ at $0$ and $K$. Therefore, because of Assumption~\ref{assump:angles_not_too_small}, \[m^2(x) \geq \min\left( f(\varepsilon), f\left( \frac{\pi}{2}- \varepsilon \right) \right) >0.\qedhere\]
\end{proof}

\paragraph{Laplacian, Green function and its asymptotics}

We define the potential of the random walk by
\begin{equation}
    U(x,y) = \mathbb{E}_x \left[ \sum_{n=0}^{+\infty} \mathds{1}_{X_n =y} \right] = \sum_{n=0}^{+\infty} \mathbb{P}_x\left( X_n =y \right).
\end{equation}
Because of \eqref{eq:proba_kill_bounded_away_from_zero}, $U(x,y) \leq \sum_{n=0}^{+\infty} (1-\delta)^n < +\infty$.

The potential of a random walk is often called its \emph{Green function}. However, the authors of \cite{BdTR} define the Green function from the point of view of graphs, that is, the inverse of the (massive) Laplacian of the graph. Looking at these functions as matrices indexed by vertices of $G$, it comes down to taking the product of the potential with a diagonal matrix. More explicitly, the \emph{Green function} $\gr$ of the graph is 
\begin{equation}\label{eq:lien_green_potential}
    \gr(x,y) = \frac{U(x,y)}{\sum_{z \in V} \rho(yz) + m^2(y)}.
\end{equation}
The matrix product formulation is 
\begin{equation}\label{eq:lien_green_potential_matrix}
    \gr = UD^{-1}
\end{equation}
where $D$ is the diagonal matrix
\begin{equation}\label{eq:diagonal_matrix_potential_green}
    D = \diag \left( \sum_{z \in V} \rho(xz) + m^2(x) \right)_{x \in V}.
\end{equation}

However, in terms of asymptotics, there will be no significant difference between $U$ and $\gr$. Indeed, Lemma~\ref{lem:diagonal_laplacian_bound} ensures that 
\begin{equation}\label{eq:encadrement_Gr_V}
    c \gr \leq U \leq c' \gr.
\end{equation}

One of the main results of \cite{BdTR} is to provide an integral formula for the Green function, which leads to an asymptotic equivalent when $y$ goes to infinity. This integral formula involves the following discrete exponential function.
\begin{defi}
    Let $x,y \in G$ and $x = x_1, \ldots, x_n=y$ a path from $x$ to $y$. We note $\overline{\alpha}_j$ the angle associated with the edge $x_j x_{j+1}$, as the angle $\overline{\alpha}_e$ in Figure~\ref{fig:angles_lozanges}. Then the exponential function $\e_{(x,y)}$ is defined by
    \[\forall u \in \C,~~\e_{(x,y)} (u) = \prod_{j=1}^{n-1} \left( i \sqrt{k'} \elliptsc\left( \frac{u-\alpha_j}{2} \right) \right).\]
    It does not depend on the path used to link $x$ to $y$.
\end{defi}

Using the previously mentioned integral formula for the Green function and the saddle-point method, the authors of \cite{BdTR} obtain the following theorem.

\begin{thm}[\cite{BdTR}, Thm 14]\label{thm:asymp_green}
    We assume that the graph is quasicrystalline and denote $|y-x|$ the graph distance between vertices $x$ and $y$ in $G^\diamond$, which is also the $\ell^1$ distance in the monotone surface of $\Z^d$ whose projection is $G^\diamond$, see after Definition~\ref{def:quasicrystalline}.
    There exist a function $\chi$ and $u_0$ (that depend on $x_0,y$) such that $\chi''(u_0) >0$ and, when the distance $|y - x_0|$ goes to infinity, 
    \begin{equation}\label{eq:decroissance_green}
        \gr(x_0,y) = \frac{k' e^{|y-x_0| \chi(u_0)}}{2\sqrt{2 \pi |y - x_0| \chi''(u_0)}}(1 + o(1)).
    \end{equation}
    Besides, when $y$ goes to infinity, the convergence of $u_0$ is implied by the convergence of the reduced coordinates of $y$, that is the $\ell^1$ direction of $y$ when lifted to $\Z^d$.

    In \eqref{eq:decroissance_green}, the decay is actually exponential. More precisely, \cite[Lemma 16]{BdTR} proves that 
    \begin{equation}\label{eq:expo_decay_inequality}
    \chi(u_0) \leq \log\left( k' \elliptnd\left( \frac{\varepsilon}{2} \right) \right) < 0.
    \end{equation}
\end{thm}

The expression of $\chi$ will be given when necessary.

The inverse of the Green function is the massive Laplacian (this is even the definition of the Green function in \cite{BdTR}).

\begin{defi}\label{def:laplacien_massique}
    The massive Laplacian $\Delta^m$ acts on functions $f : V \to \C$ by
    \begin{align}
        \left(\Delta^m f\right)(x) &= \sum_{y \in V} \rho(xy) (f(x) - f(y)) + m^2(x) f(x) \\
        &= f(x) \left( \sum_{y \in V} \rho(xy) + m^2(x)\right) - \sum_{y \in V} \rho(xy) f(y).
    \end{align}
    In terms of matrices indexed by $V$, it means that 
    \begin{equation}
        \Delta^m_{x,y} = \left(\Delta^m \delta_y\right)(x) = \left\{ \begin{array}{ll}
            - \rho(xy) & \text{if $x \neq y$,} \\
             \sum_{z \in V} \rho(xz) +m^2(x) & \text{if $x = y$.}
        \end{array} \right.
    \end{equation}
\end{defi}

Since the graph $G$ is locally finite, all the sums written before are finite. This ensures that the massive Laplacian is well defined. It is clearly symmetric.

\section{Leaky sandpile model on isoradial graphs}\label{sec:sandpile_model}
In this section, we define properly the leaky Abelian sandpile model on the graph $G$. Up to the author's knowledge, it is the first time that this model is studied on isoradial graphs. With the help of the odometer function, we define the shape of the final stable configuration. Applying a modified version of the Laplacian to the odometer function, we obtain Proposition~\ref{prop:threshold}, that bounds the boundary of the shape of the sandpile between two level curves of the Green function.

\subsection{Description of the model}
On any locally finite graph with conductances and masses, we can define a leaky sandpile model as follows. A sandpile configuration on $G=(V,E)$ is a function $s : V \to [0, + \infty)$.

\begin{model}[Leaky Abelian Sandpile Model (LASM) on $G$]\label{model:sandpile_article3}
We fix an origin vertex $x_0 \in V$. Let $N \in \N$.
    \begin{itemize}
        \item The initial configuration is $s_0 := N \delta_{x_0}$. It consists of $N$ grains of sand at the origin.
        \item A sandpile configuration $s$ is said to be \emph{stable} at site $x \in V$ if: \[s(x) < \sum_{y \in V} \rho(xy) + m^2(x).\] We say that a configuration is stable if it is stable at every site. Otherwise, we say it is unstable.
        \item As long as a configuration is unstable at site $x$, it can topple, which means:
        \begin{align*}
        &s(x) \leftarrow s(x) -  \sum_{y \in V}\rho(xy) - m^2 (x),\\
        &\forall y \in V \setminus \{x\},\quad  s(y) \leftarrow s(y) + \rho(xy).
        \end{align*}
        \item Sites topple until the configuration eventually becomes stable.
    \end{itemize}
\end{model}

The mass $m^2(x)$ is the amount of sand lost each time the site $x$ topples.

A classical property of the model is that, if at some point in the toppling process, several sites can topple, then the order chosen does not modify the eventual stable configuration (see \cite{Hol_survey}).

\begin{defi}
    The odometer function is the function $u : V \to \R_{\geq 0}$ defined by 
    \begin{equation}
        u(x) = \text{total amount of sand emitted from $x$ until stabilization}.
    \end{equation}
    It takes into account both the sand sent to neighbors \emph{and} the sand that disappeared due to leakiness.
\end{defi}

We use the odometer function to define the shape of the final configuration of the sandpile.

\begin{defi}
    We say that a vertex $x \in V$ is in the shape of the final stable configuration of the sandpile if $u(x) >0$.
\end{defi}

\subsection{Link between the shape of the sandpile and the Green function}

In order to combine the Laplacian with the odometer function, we need to slightly modify the Laplacian so that it involves \emph{proportions of sand emitted} instead of \emph{amounts} of sand, which comes down to using the Laplacian of the random walk instead of the Laplacian $\Delta$ of the graph.

\begin{defi}\label{sandpile_isoradial:laplacian}
    We define an operator $T$ that acts on functions $f : V \to \C$ by 
    \begin{align}\label{eq:operator_T_def}
        (Tf)(x) &= \sum_{y \in V} \frac{\rho(xy)}{\sum_{z \in V} \rho(yz) +m^2(y)} f(y) - f(x)\\ 
        &= \sum_{y \in V} \mathbb{P}(y \to x) f(y) - f(x).
    \end{align}
    In terms of matrices indexed by $V$, it means that 
    \begin{equation}
        T_{x,y} = \left\{ \begin{array}{ll}
            \frac{\rho(xy)}{\sum_{z \in V} \rho(yz) + m^2(y)} & \text{if $x \neq y$,}  \\
             -1&\text{if $x = y$.} 
        \end{array} \right.
    \end{equation}
    It comes down to right-multiplying the Laplacian of the graph by the diagonal matrix $-D^{-1}$ where $D$ was introduced in \eqref{eq:diagonal_matrix_potential_green}.
\end{defi}

As a direct consequence of the link between the Green function and the Laplacian, we get the following property for the potential $U$ and the operator $T$.

\begin{prop}\label{prop:operators_inverse}
    The operators $T$ and the potential $U$ satisfy $TU^{\intercal} = U^{\intercal}T = - \id$.
\end{prop}
\begin{proof}
    We use $\gr = \Delta^{-1}$, $T = -\Delta D^{-1}$, $\gr = UD^{-1}$ and the symmetry of $D$ and $\Delta$ to get
    \begin{equation*}
        T^{-1} = \left( -  \Delta D^{-1} \right)^{-1} 
        = -D \Delta^{-1}
        = -D \left(\Delta^{-1}\right)^\intercal
        = - D \gr^\intercal
        =-D D^{-1} U^\intercal
        = -U^\intercal.\qedhere
    \end{equation*}
\end{proof}

The following result links the twisted Laplacian $T$ with the odometer function $u$ and the initial and final configurations. Although its proof is very simple, it is a key result to establish the link between the shape of the sandpile and the asymptotics of the Green function.
We fixed the initial configuration to be $N \delta_{x_0}$, but the result is true for any finite initial configuration, so we decide to state it in full generality. 
\begin{prop}\label{prop:operator_applied_to_odometer}
    Let $s_0 : V \to [0, + \infty)$ be any initial configuration with a finite total amount of sand. Let $u$ and $f$ be the odometer function and the final configuration obtained from $s_0$. One has 
    \begin{equation}
        Tu = f - s_0.
    \end{equation}
\end{prop}

\begin{proof}
    By definition \eqref{eq:operator_T_def} of $T$, 
    \[(Tu)(x) = \sum_{y \in V} \frac{\rho(xy)}{\sum_{z \in V} \rho(yz) +m^2(y)} u(y) - u(x).\]
    But $\frac{\rho(xy)}{\sum_{z \in V} \rho(yz) +m^2(y)}$ is the proportion of the sand lost by $y$ that goes to $x$ each time $y$ topples. Therefore, multiplying by the total amount $u(y)$ lost by $y$, we see that $\frac{\rho(xy)}{\sum_{z \in V} \rho(yz) +m^2(y)} u(y)$ is the total amount of sand sent by $y$ to $x$, so $\sum_{y \in V} \frac{\rho(xy)}{\sum_{z \in V} \rho(yz) +m^2(y)} u(y)$ is the total amount received by $x$. In conclusion, $(Tu)(x)$ is the difference between what $x$ received and what it sent, that is, the difference between its final amount of sand and its initial amount of sand.
\end{proof}

In order to prove the main result of this section, that is Proposition~\ref{prop:threshold}, we need the following lemma.

\begin{lemma}\label{lem:bornage_somme_potentiel}
    There exists a constant $a>0$ such that, for every $x \in V$,
    \begin{equation}
        \sum_{y \in V} U(y,x) \leq a.
    \end{equation}
\end{lemma}

\begin{proof}
    Using the inequalities of \eqref{eq:encadrement_Gr_V} and the symmetry of $\gr$, we get \[U(y,x) \leq \frac{c'}{c}U(x,y).\]
    Therefore, it is enough to bound $\sum_{y \in V} U(x,y)$ instead.
    But 
    \begin{align*}
        \sum_{y \in V} U(x,y) &= \sum_{y \in V} \sum_{n=0}^{+ \infty} \mathbb{P}_x(X_n = y) \\
        &= \sum_{n=0}^{+ \infty} \sum_{y \in V} \mathbb{P}_x(X_n = y) \\
        &= \sum_{n = 0}^{+\infty} \mathbb{P}_x\left( \text{$\left(X_k\right)_{k \geq 0}$ was not already killed at step $n$} \right) \\
        & \leq  \sum_{n = 0}^{+\infty} \left(1 - \delta \right)^n,
    \end{align*}
    where the last term does not depend on $x$ and comes from Lemma~\ref{lem:diagonal_laplacian_bound}. 
\end{proof}

\begin{prop}\label{prop:threshold}
    There exist constants $\alpha, \beta >0$ such that: 
    \begin{itemize}
            \item if $\gr(x_0,x) > \frac{\alpha}{N}$, then $x$ is in the shape of the final configuration;
            \item if $\gr (x_0,x)< \frac{\beta}{N}$, then $x$ is not in the shape of the final configuration.
        \end{itemize}
\end{prop}
\begin{proof}
    Because of the bounds of \eqref{eq:encadrement_Gr_V}, it is enough to prove the same result with the potential $U$ instead of the Green function $\gr$.

    Using Propositions~\ref{prop:operators_inverse} and \ref{prop:operator_applied_to_odometer}, we obtain
    \[T \left( \frac{1}{N} u -U(x_0, \cdot)  \right) = \frac{1}{N}f.\]
    In the final configuration, no vertex $x$ has more than the threshold $\sum_{y \in V} \rho(xy) + m^2(x)$ grains of sand, therefore, by Lemma~\ref{lem:diagonal_laplacian_bound}, 
     \[0 \leq T \left( \frac{1}{N} u -U(x_0, \cdot) \right) \leq \frac{c'}{N} \mathds{1}\]
     where $\mathds{1} : x \in V \mapsto 1$.
     Left-multiplying by $T^{-1} = - U^\intercal$, which switches inequalities by negativity, we get
     \[\forall x \in V,\quad 0 \geq \frac{1}{N} u(x) - U(x_0,x) \geq - \frac{c'}{N} \left(U^\intercal\mathds{1}\right)(x).\]
     But $\left(U^\intercal\mathds{1}\right)(x) = \sum_{y \in V} U(y,x)$, so $\left(U^\intercal\mathds{1}\right)(x) \leq a$ by Lemma~\ref{lem:bornage_somme_potentiel}.
     Therefore, setting $\alpha = c'a >0$, we get
     \begin{equation}\label{eq:threshold_proof}
     \forall x \in V,\quad 0 \geq \frac{1}{N} u(x) - U(x_0,x) \geq -\frac{\alpha}{N}.
     \end{equation}
     
     If $x$ is not in the shape of the final configuration, then $u(x) = 0$, hence \eqref{eq:threshold_proof} yields $U(x_0,x) \leq \frac{\alpha}{N}$. Conversely, if $x$ is in the shape, then it has emitted sand at least once, therefore \[u(x) \geq \sum_{y \in V} \rho(xy) + m^2(x) \geq c,\] where the last inequality comes from Lemma~\ref{lem:diagonal_laplacian_bound}, so \eqref{eq:threshold_proof} yields $U(x_0,x) \geq \frac{c}{N}$.
\end{proof}

\section{Limit shapes of the sandpile}\label{sec:limit_shape}

From now on, we will only work in the quasicrystalline case, which was introduced in Definition~\ref{def:quasicrystalline}, so that we can use Theorem~\ref{thm:asymp_green}. 

\begin{assump}
    The graph is quasicrystalline.
\end{assump}

\subsection{Limit shape in $\R^d$}\label{subsec:shape_r^d}
\subsubsection{Limit shape in $\R^d$ when $N$ goes to infinity}
As mentioned earlier, under the quasicrystalline assumption, the graph can be seen as the projection onto the plane of a monotone surface of $\Z^d$, which we denote $S$. By lifting the graph onto $S$, the distance $|\cdot|$ in the asymptotics of Theorem~\ref{thm:asymp_green} is the usual $\ell^1$ distance in $\R^d$. Up to translation, we can assume that the origin $x_0$ is the origin of $\R^d$. Besides, the function $\chi$ is explicitly known: if we denote $n = (n_1(y), \ldots, n_d(y))$ the coordinates of $y$ in $\Z^d$, then \cite[Eq. (24)]{BdTR} gives
\begin{equation}
    \|y\|_1 \chi(u) = \sum_{j=1}^d n_j(y) \log \left( \sqrt{k'} \elliptnd\left( \frac{u-\alpha_j}{2} \right) \right) = -n \cdot \theta(u)
\end{equation}
where 
\begin{equation}\label{eq:def_theta}
    \theta(u) = -\left( \log \left( \sqrt{k'} \elliptnd\left( \frac{u-\alpha_1}{2} \right) \right), \ldots, \log \left( \sqrt{k'} \elliptnd\left( \frac{u-\alpha_d}{2} \right) \right) \right) .
\end{equation}
The asymptotics of Theorem~\ref{thm:asymp_green} becomes 
\begin{equation}\label{eq:decroissance_green_origin}
        \gr(x_0,y) = \frac{k' e^{-n(y) \cdot \theta(u_0)}}{2\sqrt{2 \pi \|n(y)\|_1 \chi''(u_0)}}(1 + o(1)),
    \end{equation}
and \cite{BdTR} proves that $u_0$ depends continuously on the direction $ \widehat{y} =\frac{n(y)}{\|n(y)\|_1}$ and will therefore be noted $u_{\hat{y}}$.

In that setting, we will get a limit shape of the sandpile in $\R^d$. Indeed, \cite[Section 5]{BBMR} proves the following.
\begin{thm}\label{thm:shape_sandpile}
    For $z \in \R^d \setminus \{0\}$, we denote $\widehat{z} = \frac{z}{\|z\|}$.
    Let $\alpha, \beta >0$ and $f : \R^d \to \R$ such that \[f(z) = c\left( \widehat{z} \right) \|z\|^{-\frac{1}{2}} e^{-z \cdot \gamma\left( \widehat{z} \right)}(1+o(1))\]
    as $\|z\|$ goes to infinity, where $c, \gamma : \mathbb{S}^{d-1} \to \R$ are continuous functions, $c$ is positive and $z \cdot \gamma\left( \widehat{z} \right) >0$ for each $z \in \R^d$.
    For $s \in \mathbb{S}^{d-1}$, we denote
    \begin{equation*}
        r_{N, s} = \inf \left\{r>0~|~ f(rs)  \leq \frac{\alpha}{N} \right\}\quad\text{and}\quad
        R_{N, s} = \sup \left\{ r>0~|~ f(rs) \geq \frac{\beta}{N} \right\}.
    \end{equation*}
    As $N$ goes to infinity,
    \begin{equation}
        \lim_{N \to +\infty}\frac{r_{N,s}}{\log N} = \lim_{N \to +\infty} \frac{R_{N,s}}{\log N} = \frac{1}{\gamma(s) \cdot s}
    \end{equation}
    and these limits are uniform in $s \in \mathbb{S}^{d-1}$.

    This implies that for any sequence of sets $\left(E_N\right)_{N \in \N}$ such that for every $N \in \N$,
    \[\left\{ rs~|~ s \in \mathbb{S}^{d-1},~ 0 \leq r \leq r_{N,s}\right\} \subset E_N \subset \left\{ rs~|~ s \in \mathbb{S}^{d-1},~ 0 \leq r \leq R_{N,s}\right\},\]
    $\frac{1}{\log N} E_N$ converges to 
    \begin{equation}\label{eq:limit_set}
        E :=\left\{ rs~|~ s \in \mathbb{S}^{d-1},~ 0 \leq r_s \leq \frac{1}{\gamma(s) \cdot s} \right\}
    \end{equation}
    for the Hausdorff distance.
\end{thm}

Let us remind briefly how this result is obtained. We restrict $f$ to a fixed direction $s \in \mathbb{S}^{d-1}$. If $f$ is replaced by its asymptotic equivalent, then it is invertible in direction $s$ and its inverse is easily expressed with a classical special function, the Lambert W function, which asymptotics are well known. By evaluating these asymptotics at $\frac{\alpha}{N}$ and $\frac{\beta}{N}$, we obtain the expected limit $\frac{1}{\gamma(s) \cdot s}$, and \cite[Section 5]{BBMR} proves that it is still true for the actual function $f$ instead of its asymptotics.

In our context, $f$ is the Green function and the exponential decay $\gamma$ is $\widehat{z} \mapsto \theta\left( u_{\hat{z}} \right)$, where $\theta$ and $u_{\hat{z}}$ are defined in \eqref{eq:def_theta} and the following lines. Seeing the Green function as a function defined on $S$,
Proposition~\ref{prop:threshold} shows that, when lifted onto $S$, the sandpile is between the sets 
\[ S \cap\left\{ rs~|~ s \in \mathbb{S}^{d-1},~ 0 \leq r \leq r_{N,s}\right\} \quad \text{and}  \quad S \cap \left\{ rs~|~ s \in \mathbb{S}^{d-1},~ 0 \leq r \leq R_{N,s}\right\}.\]
As a consequence of Theorem~\ref{thm:shape_sandpile}, the limit shape of the sandpile in $\R^d$ is a \emph{subset} of the set $E$ defined in \eqref{eq:limit_set}. However, since we take the intersection with $S$, the limit shape may not be the whole of $E$, as some directions could be lacking in $S$ at infinity. This motivates the following definition of the directions that appear on $S$ at infinity.

\begin{defi}\label{defi:direction_admissible}
    A direction $s \in \mathbb{S}^{d-1}$ is called \emph{admissible} for the surface $S$ if there exists a sequence $\left(y_n\right)_{n \in \N}$ of $S$ such that $\|y_n\|_1 \xrightarrow[n \to + \infty]{}+\infty$ and $\frac{y_n}{\|y_n\|_1} \xrightarrow[n\to +\infty]{}s$.
\end{defi}

A direction is admissible if it appears at infinity on the monotone surface $S$. In other words, if $s$ is admissible, then there are points on the half-line $\R_+s$, arbitrarily large, that are close to $S$.  However, it might happen that, when moving along the half-line $\R_+ s$, we oscillate between being sometimes close to the surface and sometimes far from it. This means that after normalization by $\log N$, we oscillate between seeing the direction $s$ in $S$ and not seeing it as $N$ grows, making it impossible to get a limit shape in spherical coordinates as in \eqref{eq:limit_set}. To avoid these phenomenon, we will need to make regularity assumptions about the surface. We propose two different ones. The first one states that the distance between a point $ts$ and the surface cannot grow linearly in $t$ as $t$ goes to infinity. The second one, which is stronger, states that points from $\R_+ s$ must stay at bounded distance from the surface $S$.

\begin{assump}[Regularity assumption, weak version]\label{assump:regularity_weak}
     For every $\delta>0$ and every admissible direction $s$, there exists $T >0$ such that for $t >T$, $\mathrm{d}(ts, S) \leq \delta t$.
\end{assump}

\begin{assump}[Regularity assumption, strong version]\label{assump:regularity_strong}
    There exists a constant $c >0$ such that for every admissible direction $s$ and every $t>0$, $\mathrm{d}(ts, S) \leq c$.
\end{assump}

It is clear that the strong version implies the weak version (take $T = \frac{c}{\delta}$).

Combining Theorem~\ref{thm:shape_sandpile} and the regularity assumptions, we obtain convergence results for the shape of the sandpile on the monotone surface $S$ when the number $N$ of grains of sand tends to infinity. To get convergence, we need to do a proper normalization. As hinted by Theorem~\ref{thm:shape_sandpile}, this normalization is of order $\log N$, which is typical for the leaky sandpile model (see \cite{AM22, BBMR}), whereas for the non-leaky sandpile model on $\Z^d$, the normalization is in $N^{\frac{1}{d}}$.

\begin{thm}[Convergence under the weak regularity assumption]\label{thm:case_N_infinity_weak_regularity_assumption}
    We work under Assumption~\ref{assump:regularity_weak}.
    We denote $S_N \subset S \subset \R^d$, the shape of the sandpile started with $N$ grains of sand at $x_0$ (although the choice of $x_0$ does not change the limit shape). As $N$ goes to infinity, $\frac{1}{\log N} S_N$ converges to 
    \begin{equation}\label{eq:limit_shape_N_infinity}
        \left\{ r s ~\left|~ s \textnormal{ is admissible},~ 0 \leq r \leq \frac{1}{\theta(u_s) \cdot s} \right.\right\}
    \end{equation}
in the sense that: \[\forall s \textnormal{ admissible}, ~~\forall 0 \leq r \leq \frac{1}{\theta(u_s) \cdot s},~~\mathrm{d}\left( rs, \frac{1}{\log N} S_N \right) \xrightarrow[N \to +\infty]{}0.\]
\end{thm}
\begin{proof}
    Let $s \in \mathbb{S}^{d-1}$ admissible and $0 \leq r < \frac{1}{\theta(u_s) \cdot s}$. Let $\delta >0$ such that \[\delta < \mathrm{d}\left( rs, \left\{ \left. \frac{1}{\theta(u_{\tilde s}) \cdot \tilde s} \tilde s~\right|~\tilde s \textnormal{ is admissible} \right\}\right).\] Since $\frac{r_{N,s}}{\log N}$ converges uniformly to $\frac{1}{\theta(u_s) \cdot s}$, we have \[(\log N) B(rs, \delta) \subset \left\{ \tilde{r} \tilde s~|~ \tilde s \in \mathbb{S}^{d-1},~ 0 \leq \tilde{r} \leq r_{N,\tilde s}\right\}\] for $N$ large enough, where $B(\cdot, \cdot)$ are the balls with respect to the distance of Assumption~\ref{assump:regularity_weak}.
    But \[S \cap \left\{ \tilde{r} \tilde s~|~ \tilde s \in \mathbb{S}^{d-1},~ 0 \leq \tilde{r} \leq r_{N,\tilde s}\right\}\] is a subset of the sandpile started with $N$ grains of sand, as a consequence of Proposition~\ref{prop:threshold} and by definition of $r_{N,s}$. Since $s$ is admissible and by the regularity assumption, there are elements $y_N \in S \cap (\log N) B(rs, \delta)$ for $N$ large enough. These elements are thus in the shape of the sandpile started with $N$ grains of sand, so $\frac{y_N}{\log N}$ is in the rescaled sandpile. In conclusion, $B(rs, \delta)$ contains elements of the rescaled sandpile for $N$ large enough, so
    \[\limsup_{N \to + \infty} \mathrm{d}\left( rs, \frac{1}{\log N} S_N \right) \leq \delta.\]
    This is true for any $\delta$ small enough, so $\mathrm{d}\left( rs, \frac{1}{\log N} S_N \right) \xrightarrow[N \to +\infty]{}0$.
\end{proof}

\begin{thm}[Convergence under the strong regularity assumption]\label{thm:case_N_infinity_strong_regularity_assumption}
    Under Assumption~\ref{assump:regularity_strong}, the convergence to \eqref{eq:limit_shape_N_infinity} is uniform, in the sense that \begin{equation}
        \sup_{s \textnormal{ admissible},~~ 0 \leq r \leq \frac{1}{\theta(u_s) \cdot s}} \mathrm{d}\left(rs, \frac{1}{\log N} S_N \right) \xrightarrow[N \to + \infty]{}0.
    \end{equation}
\end{thm}

\begin{proof}
    By a continuity and compactness argument, \[\inf_{\tilde s \in \mathbb{S}^{d-1}} \frac{1}{\theta(u_{\tilde s}) \cdot \tilde s} = \min_{\tilde s \in \mathbb{S}^{d-1}} \frac{1}{\theta(u_{\tilde s}) \cdot \tilde s} >0.\]
    Let $0 < \varepsilon < \displaystyle \min_{\tilde s \in \mathbb{S}^{d-1}} \frac{1}{\theta(u_{ \tilde s} ) \cdot  \tilde s}$,
    $s \in \mathbb{S}^{d-1}$ admissible and $0 \leq r \leq \frac{1}{\theta(u_s) \cdot u_s} $. We set $r' = \max\left(0, r-\varepsilon \right)$. Using Assumption~\ref{assump:regularity_strong}, for every $N$, there exists a vertex $x_N$ on the monotone surface $S$ such that $\left| \log (N) r' s - x_N \right| \leq c$, that is $\left| r's - \frac{1}{\log N}x_N \right| \leq \frac{c}{\log N}$. 
    Let \[\delta = \mathrm{d} \left( \left\{ \tilde r \tilde s~|~ \tilde s \in \mathbb{S}^{d-1},~0\leq \tilde r \leq \frac{1}{\theta(u_{\tilde s}) \cdot \tilde s} - \varepsilon \right\} , \left\{ \tilde r \tilde s~|~ \tilde s \in \mathbb{S}^{d-1},~ \tilde r \geq \frac{1}{\theta(u_{\tilde s}) \cdot \tilde s} - \frac{\varepsilon}{2} \right\} \right), \]
    which is (strictly) positive since it is the distance between a compact and a closed set that are disjoint. Let $N_0 = e^{\frac{c}{\delta}}$, so that $\frac{c}{\log N} < \delta$ for $N > N_0$. Then when $N > N_0$, since $r' s \in \left\{ \tilde r \tilde s~|~ \tilde s \in \mathbb{S}^{d-1},~\tilde r \leq \frac{1}{\theta(u_{\tilde s}) \cdot \tilde s} - \varepsilon \right\}$ and $\left| r' s - \frac{1}{\log N}x_N \right| < \delta$, we get
    \[\frac{1}{\log N}x_N \in \left\{ \tilde r \tilde s~|~ \tilde s \in \mathbb{S}^{d-1},~ \tilde r < \frac{1}{\theta(u_{\tilde s}) \cdot \tilde s} - \frac{\varepsilon}{2} \right\}.\]
    As a consequence of Theorem~\ref{thm:shape_sandpile}, since $\frac{r_{N,s}}{\log N}$ converges to $\frac{1}{\theta(u_s) \cdot s}$ uniformly in $s$, there exists $N_1$ such that for $N > N_1$, $\frac{r_{N,s}}{\log N} \geq \frac{1}{\theta(u_s) \cdot s} - \frac{\varepsilon}{2}$ for every $s$. Therefore, for $N > \max(N_0, N_1)$, $x_N \in \left\{ rs~|~ s \in \mathbb{S}^{d-1},~0 \leq r <r_{N,s} \right\}$, which implies that $x_N$ lies in the sandpile $S_N$ started with $N$ grains of sand, because of Proposition~\ref{prop:threshold} and by definition of $r_{N,s}$. Let $N_2 = e^{\frac{c}{\varepsilon}}$, so that when $N > N_2$, $\frac{c}{\log N} \leq \varepsilon$. When $N > \max(N_0, N_1, N_2)$, we have $x_N \in S_N$ and \[\left|x_N - rs \right| \leq \left| x_N - r' s \right| + \left| r' s - rs \right| \leq 2 \varepsilon.\]
    In conclusion, when $N > \max(N_0, N_1, N_2)$, $\displaystyle \sup_{s \textnormal{ admissible},~~ 0 \leq r \leq \frac{1}{\theta(u_s) \cdot s}} \mathrm{d}\left(rs, \frac{1}{\log N} S_N \right) \leq 2 \varepsilon$.
\end{proof}

\subsubsection{Limit shape in $\R^d$ when $k$ goes to $0$}

In this paragraph, we study the limit shape when $k$ goes to zero. This comes down to approaching the critical model studied by Kenyon in \cite{kenyon_laplace_dirac_planar}. Indeed, as $k$ goes to zero, $\elliptsc$ becomes $\tan$, while $\elliptdc$ becomes $\frac{1}{\cos}$, so that the formulas \eqref{eq:conductance} and \eqref{eq:mass} show that $\rho_e$ becomes $\tan(\theta_e)$ and $m^2(x)$ becomes $0$, which is the case in \cite{kenyon_laplace_dirac_planar}. Since the mass goes to $0$, it means that we are approaching the non-leaky case, that is, the classical Abelian sandpile model. However, as noted in \cite{BBMR} in the case of $\Z^d$, the limit shape obtained will not necessarily be the shape of the classical model.

Note that the order of the limits we take is first to have $N$ go to infinity and then $k$ to $0$. Therefore, according to Theorems~\ref{thm:case_N_infinity_weak_regularity_assumption} and \ref{thm:case_N_infinity_strong_regularity_assumption}, we study the limit of $\theta(u_s)$ when $k$ goes to $0$. To emphasize the dependence of $u_s$ on $k$, we will now denote it $u_s(k)$. Besides, we will note $m = k^2$ as in \cite{abramowitz} and we will use both notations $k$ and $m$ for the elliptic modulus indifferently, depending on which is the most convenient. Be careful that the parameter $m$ should not be mistaken for the mass $m^2$.

\begin{lemma}\label{lem:differentiable_in_m=0}
    Let $s = (s_1, \ldots, s_d) \in \mathbb{S}^{d-1}$ be fixed. The function $m \mapsto u_s(m)$ is differentiable at $0$. In particular, a Taylor expansion shows that when $m$ goes to $0$, $u_s(m) - u_s(0) = \mathcal{O}(m)$, that is, $u_s(k) -u_s(0) = \mathcal{O}(k^2)$.
\end{lemma}
\begin{proof}
    Up to a change of orientation of edges, we can assume that $s$ is in the non-negative orthant of $\R^d$. Let $\gamma$ be an infinite minimal path that converges to direction $s$ when lifted to $\R^d$. According to \cite[Lemma 17]{BdT}, the angles $\overline{\alpha}_j$ of the edges taken by the path $\gamma$ lie in an open interval of length $\pi$, which we choose to be $\left( -\frac{\pi}{2}, \frac{\pi}{2} \right)$ up to a rotation of the graph.

    We need to look carefully at Lemma~15 of \cite{BdTR} and its proof, where $u_s(m)$ is defined (it is their $u_0$). Let \begin{equation}
        f_s(u,m) = \sum_{j=1}^d s_j \frac{\elliptsn \cdot\elliptcn}{\elliptdn}\left( \left. \frac{u- \alpha_j}{2} \right| m\right).
    \end{equation} By definition, $u_s(m)$ is the only solution to $f_s(u,m) = 0$ in a well chosen interval of length $2K - 2 \varepsilon$.
    
    We now look at the case where $m=0$. From $s_j \geq 0$ and $\overline{\alpha}_j \in \left( -\frac{\pi}{2} , \frac{\pi}{2}  \right)$, we deduce
    \begin{equation}\label{eq:technical:non_zero_denominator}
        \sum_{j=1}^d s_j \cos\left( \overline{\alpha}_j\right) >0.
    \end{equation}
    Besides, when $m= 0$, $\elliptsn$, $\elliptcn$ and $\elliptdn$ become $\sin$, $\cos$ and $1$ (\cite[(16.6)]{abramowitz}), so the trigonometric formula for $\sin(2a)$ leads to
    \begin{equation}
        f_s(u,0) = \frac{1}{2} \sum_{j=1}^d s_j \sin \left(u- \overline{\alpha}_j\right),
    \end{equation} and $u_s(0)$ is the only solution to $f_s(u,0) = 0$ in the interval $\left( - \frac{\pi}{2}, \frac{\pi}{2} \right)$, that is 
    \begin{equation}\label{eq:explicit_formula_us(0)}
        u_s(0) = \arctan\left( \frac{\sum_{j=1}^d s_j \sin\left( \overline{\alpha}_j\right)}{\sum_{j=1}^d s_j \cos\left( \overline{\alpha}_j \right)} \right)
    \end{equation} which is well defined because of \eqref{eq:technical:non_zero_denominator}.
    
    To prove that $u_s$ is differentiable at $0$, we apply the implicit function theorem to $f_s$. To do so, we need to check that $\frac{\partial f_s}{\partial u}(u_s(0), 0) \neq 0$, i.e.,
    \begin{equation}\label{eq:deriv_non_zero}
        \sum_{j=1}^d s_j \cos \left(u_s(0)- \overline{\alpha}_j\right) \neq 0. 
    \end{equation}
    Should it be equal to $0$, usual trigonometric formulas would yield \[\tan(u_s(0)) = - \frac{\sum_{j=1}^d s_j \cos\left( \overline{\alpha}_j\right)}{\sum_{j=1}^d s_j \sin \left( \overline{\alpha}_j \right)} \] (with the usual convention $\pm \infty$ if the denominator is zero), that is, with \eqref{eq:explicit_formula_us(0)}, $\tan(u_s(0)) = - \frac{1}{\tan(u_s(0))}$, which is not possible.
    We can now apply the implicit function theorem and immediately get the conclusion of the lemma.
\end{proof}

\begin{prop}\label{prop:limit_k_0_surface}
    We fix $s$. When $m$ tends to $0$,
    \begin{align}
        \theta(u_s(m)) &= \frac{m}{2} \left( \frac{1}{2}-\sin^2 \left( \frac{u_s(0)-\overline{\alpha}_1}{2} \right), \ldots, \frac{1}{2}-\sin^2 \left( \frac{u_s(0)-\overline{\alpha}_d}{2} \right) \right) + o(m)\\
        & = \frac{m}{4} \left( \cos\left(u_s(0) - \overline{\alpha}_1\right), \ldots, \cos \left(u_s(0)- \overline{\alpha}_d\right) \right)+o(m)\\
        &=: \frac{m}{4} \theta_s(0) + o(m).
    \end{align}

    In other words, the sandpile seen in $\R^d$, normalized by $\frac{4}{m} = \frac{4}{k^2}$, converges to the set 
    \begin{equation}\label{eq:limit_shape_k_to_0_Rd}
        \left\{ rs ~\left|~ s \textnormal{ is admissible, }0\leq r \leq \frac{1}{\theta_s(0) \cdot s} \right.\right\}.
    \end{equation}
\end{prop}

Note that the limit shape in \eqref{eq:limit_shape_k_to_0_Rd} is well defined. Indeed, $\theta_s(0) \cdot s \neq 0$ as a consequence of \eqref{eq:deriv_non_zero}.

\begin{proof}
    Using $\elliptnd = \frac{1}{\elliptdn}$ (\cite[(16.3.3)]{abramowitz}) and the expansion $\elliptdn(v|m) = 1 - \frac{1}{2}m \sin^2(v) + o(m)$ (\cite[(16.13.3)]{abramowitz}), we get \[ \elliptnd(v|m) = 1+\frac{1}{2}m \sin^2(v) + o\left(m\right),\]
    that is,
    \begin{equation}
        \frac{\partial \elliptnd}{\partial m}(v|0) = \frac{1}{2} \sin^2(v).
    \end{equation}
    Moreover, according to \cite[(16.16)]{abramowitz},
    \begin{equation}
        \frac{\partial \elliptnd}{\partial v}(v|0) = 0 \elliptsd(v) \elliptcd(v) = 0.
    \end{equation}
    Therefore, a Taylor expansion at point $(v_0|0)$ shows that, when $v$ tends to $v_0$ and $m$ tends to $0$,
    \begin{equation}
        \elliptnd(v|m) = 1 + \frac{m}{2} \sin^2(v_0) + o\left( |v-v_0| + |m| \right),
    \end{equation}
    hence
    \begin{equation}
        \elliptnd \left( \left. \frac{u_s(m) - \alpha_j}{2} \right| m\right) = 1 + \frac{m}{2} \sin^2 \left( \frac{u_s(0) - \overline{\alpha}_j}{2} \right) + o\left( \left|u_s(m) - u_s(0) \right| + \left| \alpha_j - \overline{\alpha}_j \right| + |m| \right)
    \end{equation}
     for $j \in \{1, \ldots,d\}$. But $u_s(m) - u_s(0) = \mathcal{O}(m)$ as a consequence of Lemma~\ref{lem:differentiable_in_m=0} and $\alpha_j -\overline{\alpha}_j = \mathcal{O}(m)$ because $m \mapsto \alpha_j = \frac{2K(m)}{\pi} \overline{\alpha}_j$ is differentiable at $0$, so we get 
     \begin{equation}
        \elliptnd \left( \left. \frac{u_s(m) - \alpha_j}{2} \right| m\right) = 1 + \frac{m}{2} \sin^2 \left( \frac{u_s(0) - \alpha_j}{2} \right) + o\left( m \right).
    \end{equation}
    
    Besides, $\sqrt{k'} = \left( 1-m \right)^{\frac{1}{4}} = 1 - \frac{m}{4} + o\left( m \right)$, so the definition \eqref{eq:def_theta} of $\theta$ gives 
    \begin{align*}
        \theta(u_s(m)|m) &= -\left( \log \left( 1 + \frac{m}{2} \left( \sin^2  \left( \frac{u_s(0)- \alpha_j}{2} \right) - \frac{1}{2}\right) +o(m) \right)  \right)_{1 \leq j \leq d}\\
        &= \frac{m}{2} \left( \frac{1}{2}- \sin^2 \left( \frac{u_s(0)- \alpha_j}{2} \right) \right)_{1 \leq j \leq d} + o \left( m \right).
    \end{align*}
    The other formulation of the asymptotic equivalent is a simple consequence of trigonometric formulas.
\end{proof}

\subsection{Limit shape on the plane for asymptotically flat graphs}\label{subsec:shape_plane}

Although the asymptotics of the Green function in Theorem~\ref{thm:asymp_green} are easily expressed with the coordinates on the monotone surface of $\Z^d$, the isoradial graph naturally lies on the plane, hence the need to study the projection from the monotone surface of $\R^d$ onto the plane.

Remind that this projection is 
\begin{equation}
    \pi(x_1, \ldots, x_d) = \sum_{j=1}^d x_j e^{i \overline{\alpha}_j}.
\end{equation}
It is a bijection between the monotone surface $S$ and the diamond graph $G^\diamond$.

We start with a regularity property of $\pi$ which is stated without a proof in \cite{BR22}.

\begin{lemma}\label{lem:bilipschitz}
    The projection $\pi : S \to G^\diamond\subset \C$ is bi-Lipschitz.
\end{lemma}
\begin{proof}
    Let $x = (x_1, \ldots, x_d) \in S$. It is clear that \[\left| \pi(x) \right| = \left| \sum_{j=1}^d x_j e^{i \alpha_j}  \right| \leq \sum_{j=1}^d |x_j| = \|x\|_1,\]
    so $\pi$ is $1$-Lipschitz.
    
    Let us move to $\pi^{-1}$. According to \cite[Lemma 17]{BdT}, there exists a minimal path from $0$ to $\pi(x)$ in $G^\diamond$ whose edges angles are all in an open interval of length $\pi$. Therefore, up to switching the direction of some types of edges and up to a rotation of the graph, we can assume that there exists a minimal path in $G^\diamond$ that takes $x_j$ times an edge of type $e^{i \overline{\alpha}_j}$ where $x_1, \ldots, x_d \geq 0$ and that $\overline{\alpha}_1, \ldots, \overline{\alpha}_d \in \left( - \frac{\pi}{2} , \frac{\pi}{2} \right)$. Let $\delta = \min \left( \cos\left(\overline{\alpha}_j\right) \right)_{1 \leq j \leq d} >0$. Then for every $j$, the projection of $e^{i \overline{\alpha}_j}$ onto the horizontal axis is at least $\delta$, so \[\left|\sum_{j=1}^d x_j e^{i \alpha_j} \right| \geq \delta \sum_{j=1}^d \left|x_j\right|,\]
    that is \[ \left|\pi(x) \right| \geq \delta \left\| x \right\|_1. \qedhere\] 
\end{proof}

Given the expression in spherical coordinates of \eqref{eq:limit_shape_N_infinity}, we seek for a notion of convergence on the plane for sequences of $\Z^d$ that go to infinity along an admissible direction of the monotone surface of $\Z^d$.

However, going to infinity in a direction on the plane does not necessarily imply going to infinity in a direction on the monotone surface of $\Z^d$. See Figure~4 in \cite{BR22} for an illustration. To avoid this kind of phenomenon, the authors of \cite{BR22} introduced the notion of asymptotically flat isoradial graphs.

\begin{defi}\label{defi:asymptotically_flat}
     We assume that the isoradial graph $G$ is quasicrystalline with $d$ directions for its diamond graph and we write $n(v) = (n_1(v), \ldots,n_d(v)) := \pi^{-1}(v)$ the coordinates of a vertex $v$ in the monotone surface $S$ of $\Z^d$. The graph is said to be asymptotically flat if the convergence of a sequence $\left(v_k\right)_{k \in \N} \in \left(G^\diamond\right)^\N$ to infinity in a direction $\widehat{v}$ in the unit circle $\mathbb{S}^1$ of $\R^2$ implies the convergence of the proportions $\frac{n_j\left( v_k \right)}{\left\| n(v_k) \right\|_1}$ to a limit $n\left( \widehat{v} \right) = \left(n_1\left( \widehat{v} \right), \ldots, n_d\left( \widehat{v} \right)\right)$ that depends only on the direction $\widehat{v}$.
\end{defi}

\begin{assump}\label{assump:asymptotically_flat}
    From now on, the graph considered is assumed to be asymptotically flat.
\end{assump}

Note that, conversely, if a sequence $\left( y_k \right)_{k \in \N} \in S^\N$ tends to infinity in a way such that $\frac{y_k}{\|y_k\|_1} \xrightarrow[k \to +\infty]{}s \in \mathbb{S}^{d-1}$, then using $y_k = \|y_k\|_1 s + o(y_k)$ and the linearity and continuity of $\pi$, we get \[\frac{\pi(y_k)}{|\pi(y_k)|} = \frac{1}{\left|\pi\left( \frac{y_k}{\|y_k\|_1}\right)\right|} \pi(s) + o(1) \xrightarrow[k \to +\infty]{}\frac{\pi(s)}{|\pi(s)|},\] so the convergence in a direction of $\R^d$ implies the convergence of the projection on the plane.

Under this assumption, identifying the directions that appear at infinity on the monotone surface of $\Z^d$ is easy: they are the directions in $\R^d$ corresponding to directions on the plane. Indeed, we have the following proposition.
\begin{prop}\label{prop:admissible_directions_asymp_flat}
    The admissible directions are the $n\left( \widehat{v} \right)$ for $\widehat{v}$ in the unit circle $\mathbb{S}^1$.
\end{prop}
\begin{proof}
Let $s$ be an admissible direction. There exists a sequence $\left(y_k\right)_{k \in \N}$ on the monotone surface $S$ that tends to infinity and such that $\frac{y_k}{\|y_k\|_1} \xrightarrow[k \to +\infty]{}s$. Let $v_k = \pi(y_k)$ and $\widehat{v} = \frac{\pi(s)}{|\pi(s)|}$. 
Because of Lemma~\ref{lem:bilipschitz}, there exists a constant $\eta >0$ such that for every $k$, $|v_k| = |\pi(y_k)| \geq \eta \|y_k\| \xrightarrow[k \to + \infty]{}+\infty$. Besides, from $y_k = \|y_k\| s + o (y_k)$, we get 
\begin{equation}
\frac{v_k}{|v_k|} = \frac{\pi(y_k)}{|\pi(y_k)|} = \frac{\|y_k\|_1}{|\pi(y_k)|} \pi(s) +o \left( 1 \right) = \widehat{v} + o(1)
\end{equation}
so $v_k$ tends to infinity in direction $\widehat{v}$. Finally, the proportions $\frac{n(v_k)}{\|n(v_k)\|_1} = \frac{\pi^{-1}(v_k)}{\left\| \pi^{-1}(v_k)\right\|_1} = \frac{y_k}{\|y_k\|_1}$ tend to $s$, so $n\left( \widehat{v} \right) = s$.

Conversely, let $\widehat{v} \in \mathbb{S}^1$. Since the diamond graph $G^\diamond$ covers the plane with unit rhombi, there exists $v_k \in G^\diamond$ such that $|v_k - k \widehat{v}| \leq 1$ for every $k \in \N$. We have $\frac{v_k}{|v_k|} \xrightarrow[k \to +\infty]{} \widehat{v}$, so $\frac{n(v_k)}{\|n(v_k)\|_1} \xrightarrow[k \to +\infty]{} n\left( \widehat{v} \right)$. In other words, if $y_k = \pi^{-1}(v_k) \in S$, we have $\frac{y_k}{\|y_k\|_1} \xrightarrow[k \to +\infty]{}n\left( \widehat{v} \right)$, so $n\left( \widehat{v} \right)$ is an admissible direction.
\end{proof}

In addition to the simple description of admissible directions, the regularity assumptions are automatically true for asymptotically flat graphs.

\begin{prop}
    The strong regularity assumption (Assumption~\ref{assump:regularity_strong}) is automatically true for asymptotically flat graphs.
\end{prop}
\begin{proof}
    Let $\widehat{v} \in \mathbb{S}^1$ and $t>0$. Similarly to the previous proof, since the diamond graph covers the plane with unit rhombi, there exists $v_t \in G^\diamond$ such that $\left|v_t - t \widehat{v} \right| \leq 1$. Since $\pi$ is bi-Lipschitz (Lemma~\ref{lem:bilipschitz}), $\pi^{-1} (tv)$ is at distance at most $L\left( \pi^{-1} \right)$ from $\pi^{-1}(v_t)$, where $L\left( \pi^{-1}\right)$ is a Lipschitz constant of $\pi^{-1}$. In other words, $t \pi^{-1}\left( \widehat{v}  \right)$ is at a bounded distance from a point of the monotone surface. The admissible directions are the $\pi^{-1}\left( \widehat{v} \right)$, so this proves Assumption~\ref{assump:regularity_strong} with the constant $L\left( \pi^{-1} \right)$.
\end{proof}

\subsubsection{Limit shape on the plane when $N$ goes to infinity}

As a consequence of the previous proposition, Theorem~\ref{thm:case_N_infinity_strong_regularity_assumption} applies to asymptotically flat graphs, giving a uniform convergence in $\R^d$. Since the projection $\pi$ is Lipschitz, we obtain the following result on the plane.

\begin{thm}\label{thm:limit_shape_N_plane}
    The sandpile $S_{N,\textnormal{plane}}$ on the plane normalized by $\log N$ converges uniformly to the projection of \eqref{eq:limit_shape_N_infinity}, in the sense that
    \[\sup_{\hat{v} \in \mathbb{S}^1,~~ 0 \leq r \leq \frac{1}{\theta \left( u_{n\left(\hat{v} \right)} \right) \cdot n\left(\hat{v} \right)} } \mathrm{d} \left(r\hat{v}, \frac{1}{\log N} S_{N,\textnormal{plane}}\right) \xrightarrow[N \to + \infty]{}0.\]
\end{thm}

\subsubsection{Limit shape on the plane when $k$ goes to $0$}

We finally study on the plane the regime where $k$ tends to $0$, that is the projection of \eqref{eq:limit_shape_k_to_0_Rd}. We will show the remarkable property that, independently from the geometry of the graph, we obtain a universal shape: the unit circle. Note that in the case of $\Z^d$ studied in \cite{BBMR}, the shape was not universal and could be all kinds of ellipsoids.

\begin{prop}\label{prop:limit_shape_k_0_plane}
The projection of the limit shape \eqref{eq:limit_shape_k_to_0_Rd} onto the plane is the unit circle.
\end{prop}

\begin{proof}
    Let $s \in \mathbb{S}^{d-1}$ an admissible direction and $r = \frac{1}{\theta_s(0) \cdot s}$ so that $rs$ is in the boundary of \eqref{eq:limit_shape_k_to_0_Rd}. We have 
    \begin{align*}
        \re\left( \pi(rs) \right) &= \re \left(\frac{\sum_{j=1}^d s_j e^{i \overline{\alpha}_j}}{\sum_{j=1}^d s_j \cos\left(u_s(0) - \overline{\alpha}_j\right)}\right) \\
        &= \frac{\sum_{j=1}^d s_j \cos\left( \overline{\alpha}_j\right)}{ \cos(u_s(0)) \sum_{j=1}^ds_j \cos\left( \overline{\alpha}_j \right) + \sin(u_s(0)) \sum_{j=1}^d s_j \sin\left( \overline{\alpha}_j \right)}\\
        &= \frac{1}{\cos(u_s(0)) + \sin(u_s(0)) \frac{\sum_{j=1}^d s_j \sin\left( \overline{\alpha}_j\right)}{\sum_{j=1}^d s_j \cos \left( \overline{\alpha}_j\right)}}\\
        &= \frac{1}{\cos(u_s(0)) + \sin(u_s(0)) \tan(u_s(0))} \\
        &= \frac{\cos(u_s(0))}{ \cos^2(u_s(0)) + \sin^2 (u_s(0))}\\
        &= \cos(u_s(0)).
    \end{align*}
Similarly,
\[\im (\pi(rs)) = \sin(u_s(0)),\]
so $\pi(rs)$ is in the unit circle. To conclude, \cite[Theorem~2]{BR22} shows that the circle is \emph{fully} described as $s$ varies.
\end{proof}

\section*{Acknowledgments}

The author would like to thank the authors of \cite{BdTR} for fruitful discussions, advices and explanations of some details in their proofs. 

This work was conducted within the France 2030 framework programme, Centre Henri Lebesgue ANR-11-LABX-0020-01.

During the development of this work, the author was partially supported by the DIMERS project ANR-18-CE40-0033, the RAWABRANCH project ANR-23-CE40-0008, both funded by the French National Research Agency, and by  the European Research Council (ERC) project COMBINEPIC under the European Union’s Horizon 2020 research and innovation programme under the Grant Agreement No. 759702.

\bibliographystyle{plain}
\bibliography{biblio}

\begin{thebibliography}{10}

\bibitem{abramowitz}
Milton Abramowitz and Irene~A. Stegun.
\newblock {\em Handbook of mathematical functions with formulas, graphs, and mathematical tables}, volume No. 55 of {\em National Bureau of Standards Applied Mathematics Series}.
\newblock U. S. Government Printing Office, Washington, DC, 1964.

\bibitem{AM22}
Ian Alevy and Sevak Mkrtchyan.
\newblock The limit shape of the leaky {A}belian sandpile model.
\newblock {\em Int. Math. Res. Not. IMRN}, (16):12767--12802, 2022.

\bibitem{BTW}
Per Bak, Chao Tang, and Kurt Wiesenfeld.
\newblock Self-organized criticality: An explanation of the 1/f noise.
\newblock {\em Phys. Rev. Lett.}, 59:381--384, Jul 1987.

\bibitem{BBMR}
Théo Ballu, Cédric Boutillier, Sevak Mkrtchyan, and Kilian Raschel.
\newblock Limit shape of the leaky abelian sandpile model with multiple layers, 2025.

\bibitem{Bou21}
Ahmed Bou-Rabee.
\newblock Convergence of the random abelian sandpile.
\newblock {\em Ann. Probab.}, 49(6):3168--3196, 2021.

\bibitem{BdT}
C\'edric Boutillier and B\'eatrice de~Tili\`ere.
\newblock The critical {$Z$}-invariant {I}sing model via dimers: locality property.
\newblock {\em Comm. Math. Phys.}, 301(2):473--516, 2011.

\bibitem{BdT_survey_isoradial}
C\'edric Boutillier and B\'eatrice de~Tili\`ere.
\newblock Statistical mechanics on isoradial graphs.
\newblock In {\em Probability in complex physical systems}, volume~11 of {\em Springer Proc. Math.}, pages 491--512. Springer, Heidelberg, 2012.

\bibitem{BdTR}
C\'{e}dric Boutillier, B\'{e}atrice de~Tili\`ere, and Kilian Raschel.
\newblock The {$Z$}-invariant massive {L}aplacian on isoradial graphs.
\newblock {\em Invent. Math.}, 208(1):109--189, 2017.

\bibitem{BR22}
C\'{e}dric Boutillier and Kilian Raschel.
\newblock Martin boundary of killed random walks on isoradial graphs.
\newblock {\em Potential Anal.}, 57(2):201--226, 2022.

\bibitem{ChSm}
Dmitry Chelkak and Stanislav Smirnov.
\newblock Discrete complex analysis on isoradial graphs.
\newblock {\em Adv. Math.}, 228(3):1590--1630, 2011.

\bibitem{Hol_survey}
Alexander~E. Holroyd, Lionel Levine, Karola M\'esz\'aros, Yuval Peres, James Propp, and David~B. Wilson.
\newblock Chip-firing and rotor-routing on directed graphs.
\newblock In {\em In and out of equilibrium. 2}, volume~60 of {\em Progr. Probab.}, pages 331--364. Birkh\"auser, Basel, 2008.

\bibitem{Jarai_survey}
Antal~A. J\'arai.
\newblock Sandpile models.
\newblock {\em Probab. Surv.}, 15:243--306, 2018.

\bibitem{kenyon_laplace_dirac_planar}
R.~Kenyon.
\newblock The {L}aplacian and {D}irac operators on critical planar graphs.
\newblock {\em Invent. Math.}, 150(2):409--439, 2002.

\bibitem{Law}
Derek~F. Lawden.
\newblock {\em Elliptic functions and applications}, volume~80 of {\em Applied Mathematical Sciences}.
\newblock Springer-Verlag, New York, 1989.

\bibitem{LevPegSma}
Lionel Levine, Wesley Pegden, and Charles~K. Smart.
\newblock Apollonian structure in the {A}belian sandpile.
\newblock {\em Geom. Funct. Anal.}, 26(1):306--336, 2016.

\bibitem{PegSma}
Wesley Pegden and Charles~K. Smart.
\newblock Convergence of the {A}belian sandpile.
\newblock {\em Duke Math. J.}, 162(4):627--642, 2013.

\bibitem{PrWi-98}
James~Gary Propp and David~Bruce Wilson.
\newblock How to get a perfectly random sample from a generic {M}arkov chain and generate a random spanning tree of a directed graph.
\newblock volume~27, pages 170--217. 1998.
\newblock 7th Annual ACM-SIAM Symposium on Discrete Algorithms (Atlanta, GA, 1996).

\end{thebibliography}

\end{document}